\documentclass[11pt, bezier,amstex]{article}
\usepackage{amssymb,amsmath,latexsym, bbm}
\usepackage{epsfig}

\usepackage{epic,eepic,wrapfig,color}
\usepackage{tocbibind}

\usepackage{enumerate}
\usepackage[numbers,sort&compress]{natbib}
\usepackage{amsthm}
\usepackage{ifxetex}

\ifxetex
\usepackage[xetex]{hyperref}
\else
\usepackage{hyperref}
\fi 

\newtheorem{thm}{Theorem}[section]
\newtheorem{cor}[thm]{Corollary}
\newtheorem{lemma}[thm]{Lemma}
\newtheorem{prop}[thm]{Proposition}
\newtheorem{defn}[thm]{Definition}

\numberwithin{equation}{section}

\allowdisplaybreaks[4]

\newcommand\R{{\mathbb{R}}}
\newcommand\bbE{{\mathbb{E}}}
\newcommand\bbP{{\mathbb{P}}}
\newcommand\im{{\mathrm{i}}}
\newcommand\bbD{{\mathbb{D}}}
\newcommand\bbK{{\mathbb{K}}}

\newcommand\bfQ{{\mathbf{Q}}}

\newcommand\calF{{\mathcal{F}}}
\newcommand\sL{{\mathcal{L}}}

\newcommand\e{{\mathrm{e}}}
\newcommand\rd{{\mathrm{d}}}

\newcommand\id{{\mathbbm{1}}}

\def\eps{\varepsilon}
\def\wt{\widetilde}
\def\<{\langle}
\def\>{\rangle}

\begin{document}
\title{\bf Uniqueness of Stable Processes with Drift}
\author{{\bf Zhen-Qing Chen}\thanks{Research partially supported by NSF Grant DMS-1206276, and NNSFC Grant 11128101.}
  \quad and \quad
{\bf Longmin Wang}\thanks{Research partially supported by NSFC Grant 11101222, and the Fundamental Research Funds for the Central Universities.}}

\date{September 11, 2013}
\maketitle

\begin{abstract}
Suppose that $d\geq1$ and $\alpha\in (1, 2)$.
Let $Y$ be a rotationally symmetric $\alpha$-stable process on
 $\R^d$ and $b$ a
$\R^d$-valued measurable function on $\R^d$ belonging to a certain
Kato class of $Y$. We show that $\rd X^b_t=\rd Y_t+b(X^b_t)\rd t$ with $X^b_0=x$ has a
 unique weak solution for every $x\in \R^d$.
Let $\sL^b=-(-\Delta)^{\alpha/2} + b \cdot \nabla$,
which is the infinitesimal generator of $X^b$.
Denote by $C^\infty_c(\R^d)$
the space of smooth functions on $\R^d$ with compact support.
We further show that the martingale problem
for $(\sL^b, C^\infty_c(\R^d))$ has a unique solution
 for each initial value $x\in \R^d$.
\end{abstract}



\section{Introduction}\label{sec:intro}

Throughout this paper, unless otherwise stated, $d\geq 1$ and $\alpha \in (1, 2)$.
 A rotationally symmetric
$\alpha$-stable process $Y $ in $\R^d$  is a L\'evy process
with characteristic function given by
\begin{equation}
  \bbE[\exp(\im \xi\cdot (Y_t -Y_0 ))] = \exp\left(-t
       |\xi|^ \alpha  \right),\quad \xi
  \in\R^d.
\end{equation}
The infinitesimal generator of $Y $ is the
fractional Laplacian $\Delta^{\alpha/2}:= -(-\Delta)^{\alpha/2}$.
Here
we use ``$:=$'' to denote a definition.
Denote by $B(x,r)$ the open ball in $\R^d$
centered at $x\in\R^d$ with radius $r>0$ and $\rd x$ the
Lebesgue measure on $\R^d$.

\begin{defn}
For a real-valued function $f$ on $\R^d$ and $r>0$, define
\begin{equation}
  \label{eq:kato}
  M_f^{\alpha}(r) := \sup_{x\in\R^d}
  \int_{B(x,r)}\frac{|f(y)|}{|x-y|^{d+1-\alpha}} \rd y.
\end{equation}
A function $f$ on $\R^d$ is said to belong to the \emph{Kato class}
$\bbK_{d,\alpha-1}$ if $\lim_{r\downarrow0}M_f^\alpha(r)=0$.
\end{defn}

Using H\"older's inequality, it is easy to see that for every
$p>d/(\alpha-1)$, $L^\infty(\R^d;\rd x)+ L^p(\R^d;\rd x)\subset
\bbK_{d,\alpha-1}$.  Throughout this paper we will assume
$b=(b_1, \cdots, b_d)$ is a $\R^d$-valued function on $\R^d$
such that $|b|\in \bbK_{d, \alpha-1}$.

 In this paper, we are concerned with the existence and uniqueness of
 weak solutions to following stochastic differential equation (SDE)
\begin{equation}\label{e:1.8}
\rd X^b_t=\rd Y_t +b(X^b_t) \rd t, \quad X^b_0 =x.
\end{equation}

A solution of \eqref{e:1.8}, if it exists, will be called
$\alpha$-stable process with drift $b$.
When $Y$ is a Brownian motion (which corresponds to $\alpha =2)$,
it is well known  that
Brownian motion with drift can be obtained from Brownian motion through
a change of measure called Girsanov transform. But for symmetric $\alpha$-stable
process (where $0<\alpha <2$), SDE
\eqref{e:1.8} can not be solved by a change of measure. This is because
 $Y$ is a purely discontinuous L\'evy process and so the effect of
  a Girsanov transform can only produce a purely discontinuous ``drift term";
  see \cite{CS, HWY}.

In this paper, we show that \eqref{e:1.8} has a unique weak solution
for every initial value $x$. We achieve this by showing that the corresponding
martingale problem for SDE \eqref{e:1.8} is well-posed.
 Define
$\sL^b=\Delta^{\alpha/2}+b\cdot \nabla$.
It easy to see by using Ito's formula that $\sL^b$ is the
infinitesimal generator for  solutions of \eqref{e:1.8}.
Let ${\mathbb D}([0, \infty); \R^d)$ be the space of right continuous
$\R^d$-valued functions having left limits on $[0, \infty)$, equipped
with Skorokhod topology. For $t\geq 0$, denote by $X_t$ the
projection coordinate map on ${\mathbb D}([0, \infty); \R^d)$.
A probability measure $\bfQ$ on the Skorokhod space
${\mathbb D}([0, \infty); \R^d)$ is said to
 be a solution to the martingale problem
for $(\sL^b, C^\infty_c(\R^d))$ with initial value $x\in \R^d$
if   $\bfQ(X_0=x)=1$ and for every
$f\in C^\infty_c(\R^d)$,
$\int_0^t | \sL^b f (X_s)| \rd s <\infty$
$\bfQ$-a.s. for every $t>0$ and
\begin{equation}\label{e:1.7}
 M^f_t:=f(X_t) -f (X_0)-\int_0^t \sL^b f(X_s) \rd s
\end{equation}
is a $\bfQ$-martingale.
The martingale problem for $(\sL^b, C^\infty_c(\R^d))$
with initial value $x\in \R^d$
is said to be well-posed if it has a unique solution.
The following is the main result of this paper.

\begin{thm}\label{T:1.2}
For each $x\in \R^d$,   SDE \eqref{e:1.8}
has a unique weak solution. Moreover,   weak solutions with different
starting points can all be constructed on the canonical
Skorohod space $\bbD ([0, \infty); \R^d)$ and
the symmetric $\alpha$-stable process $Y$ in \eqref{e:1.8}
can be chosen  in such a way that it is the same
for all starting point $x\in \R^d$.
The law of the unique weak solution to SDE \eqref{e:1.8} is the
unique solution to the martingale problem for $(\sL^b, C^\infty_c(\R^d))$.
\end{thm}

The unique weak solutions of \eqref{e:1.8} form a strong
Markov process $X^b$.
Theorem \ref{T:1.2} combined with the main result
of \cite{BJ2007} and \cite{CKS2012a}
readily gives sharp two-sided estimates
on the transition density $p^b(t, x, y)$ of $X^b$ as well as on
the transition density $p^b_D(t, x, y)$
of the subprocess $X^{b,D}$ of $X^b$ killed upon leaving a bounded
$C^{1,1}$ open set.
We refer the definition of $C^{1,1}$ open set and its $C^{1,1}$ characteristics to \cite{CKS2012a}.
For $x\in D$, let $\delta_{D}(x)$ denote the Euclidean distance
between $x$ and $\partial D$. The diameter of $D$ will be denoted
as ${\rm diam}(D)$.

\begin{cor}\label{C:1.3}
{\rm (i)}
 $X^b$ has a jointly continuous transition density function $p^b(t, x, y)$
with respect to the Lebesgue measure on $\R^d$.
Moreover, for every  $T>0$, there is a constant $c_1>1$
depending only on $d$, $\alpha$, $T$ and
on $b$ only through the rate at which $ M^\alpha_{|b|}(r)$ goes to zero so that for  $(t, x, y)\in (0, T]\times \R^d \times \R^d$,
$$
c_1^{-1}\left( t^{-d/\alpha} \wedge
\frac{t}{|x-y|^{d+\alpha}}\right)  \le q^b (t, x, y) \le c_1 \left(
t^{-d/\alpha} \wedge \frac{t}{|x-y|^{d+\alpha}}\right).
$$

{\rm (ii)} Let $d\geq2$ and let $D$ be a bounded $C^{1,1}$ open
subset of $\, \R^d$ with $C^{1,1}$ characteristics $(R_0,
\Lambda_0)$. Define
$$
f_D(t, x, y)=\left( 1\wedge
\frac{\delta_D(x)^{\alpha/2}}{\sqrt{t}}\right) \left( 1\wedge
\frac{\delta_D(y)^{\alpha/2}}{\sqrt{t}}\right) \left( t^{-d/\alpha}
\wedge \frac{t}{|x-y|^{d+\alpha}}\right).
$$
For each $T>0$, there are  constants $c_2=c_2(T, R_0, \Lambda_0, d,
\alpha, \text{diam}(D), b) \geq 1$ and
$c_3=c_3(T, d, \alpha, D, b) \geq 1$
with the dependence on $b$ only through the rate at
which $ M^\alpha_{|b|}(r)$ goes to zero such that
\begin{description}
\item{\rm (a)}   on $(0, T]\times D\times D$,
$
c_2^{-1} f_D(t, x, y)  \leq p^{b}_D(t, x, y)\,
\leq c_2 f_D(t, x, y);
$
\item{\rm (b)}  on $[T,
\infty)\times D\times D$,
$$
c_3^{-1}\, \e^{-t \lambda_1^{b, D}}\, \delta_D (x)^{\alpha/2}\, \delta_D
(y)^{\alpha/2} \,\leq\,  p^{b}_D(t, x, y) \,\leq\, c_3 \,
\e^{-t \lambda_1^{b, D}}\, \delta_D (x)^{\alpha/2} \,\delta_D
(y)^{\alpha/2} ,
$$
where $\lambda_1^{b, D}:=- \sup {\rm Re} (\sigma(\sL^b|_D)) >0$.
Here $\sigma (\sL^b|_D)$ denotes the spectrum of the non-local operator
$\sL^b$ in $D$ with zero exterior condition.
\end{description}
\end{cor}

Here and in the sequel,  for $a$, $b\in \R$, $a\wedge b:=\min \{a, b\}$,
$a\vee b:=\max\{a, b\}$, and the meaning of the phrase
``depending on $b$ only via the rate at which $M^\alpha_{|b|}(r)$
goes to zero'' is that the statement is true
for any $\R^d$-valued function $\wt  b$ on $\R^d$ with
$M^\alpha_{|\wt  b|} (r)\le M^\alpha_{|b|}(r)$
for all  $r>0$.

\medskip

The existence of martingale solution to $(\sL^b, C_c^\infty (\R^d))$
is established in \cite[Theorem 2.5]{CKS2012a}, using
a (particular)
fundamental solution of $\sL^b$ constructed in \cite{BJ2007}.
One deduces easily from Ito's formula that
 the uniqueness of the martingale problem for $(\sL^b, C^\infty_c(\R^d))$
implies the weak uniqueness of SDE \eqref{e:1.8}.
So the main point of Theorem \ref{T:1.2} is on the
uniqueness of the martingale problem for $(\sL^b, C^\infty_c(\R^d))$ and the
 existence of a weak solution to SDE \eqref{e:1.8}.
The novelty here is that the drift $b$ is a function in Kato class
$\bbK_{d, \alpha-1}$, which in general is merely measurable and can be unbounded.
Thus Picard's iteration method is not applicable either.
Motivated by the approach in \cite{BC2003}, we establish the
uniqueness of the solutions to the martingale problem for $(\sL^b, C^\infty_c(\R^d))$
by showing that its resolvent is uniquely determined.
This uniqueness of the martingale problem for $(\sL^b, C^\infty_c(\R^d))$
in particular gives the uniqueness of the fundamental solution to
$\sL^b$, which was not addressed in \cite{BJ2007}; see Theorem \ref{T:2.2}
below.

The equivalence between
weak solutions to SDE driven by Brownian motion and solutions to
martingale problems for elliptic operators is well known. The crucial
ingredient in
this connection is a martingale representation theorem for Brownian motion.
Such a martingale representation theorem is not available
for stable processes.
Recently, Kurtz \cite{K} studied equivalence between weak solutions
to a class of SDEs driven by Poisson random measures
and solutions to martingale problems for a class of  non-local operators
 using a non-constructive approach.
We point out that one can not deduce the existence of weak solution to
SDE \eqref{e:1.8} from the existence of the martingale problem for $(\sL^b, C^\infty_c(\R^d))$
by applying results from \cite{K} because $\sL^b f$ is typically
unbounded for $f\in C^\infty_c(\R^d)$.
In this paper, we develop a new approach to the weak existence
of solutions to  SDE \eqref{e:1.8}.
We believe this new approach is potentially useful to
study weak existence for some other SDEs with singular drifts, especially
those driven by discontinuous L\'evy processes.
Our new approach uses  the L\'evy system of the strong Markov process $X^b$
obtained  from the unique solution to the martingale
problem for $(\sL^b, C^\infty_c(\R^d))$ and stochastic calculus to
construct a weak solution to SDE \eqref{e:1.8}.

\medskip

Very recently, around the same time as the first version of this paper
was completed,
Kim and Song \cite{KS} studied stable process
with singular drift,
analogous to Brownian motion with singular drift studied
in Bass and Chen \cite{BC2003}.
Intuitively speaking, stable process with singular drift studied in \cite{KS}
corresponding to SDE \eqref{e:1.8} with $b$ being replaced by suitable measure.
However, the existence and uniqueness of the solution
in \cite{KS} is formulated in a weaker sense, as in \cite{BC2003}.
When applying to the Kato function $b$ case considered in this paper,
the results in \cite{BC2003} do not give the existence and uniqueness
of weak solutions to SDE \ref{e:1.8} nor the well-posedness of the martingale
problem for $(\sL^b, C^\infty_c (\R^d))$.

\medskip

The approach of this paper is quite robust.
It can be applied to
study some other stochastic models. For example, it can be used
to establish, for each $b=(b_1, \cdots, b_d)\in \bbK_{d, 1}$,
the well-posedness of martingale problem for $(\Delta
+b \cdot \nabla, C^\infty_c(\R^d))$ and to establish
the weak existence and uniqueness of solutions to Brownian motion
with singular drift: $\rd X_t =\rd B_t +b(X_t)\rd t$.
The rest of the paper is organized as follows.
The proof of uniqueness of the martingale problem is given in Section \ref{S:2},
while the proof of Theorem \ref{T:1.2} and Corollary \ref{C:1.3} are presented in Section \ref{S:3}.

\section{Uniqueness of martingale problem}\label{S:2}

Recall that
$\sL^b=\Delta^{\alpha/2}+b\cdot \nabla$.
When $b=0$, we simply write $\sL^0$ as $\sL$; that is, $\sL = \Delta^{\alpha/2}$.
Let  $b=(b_1, \cdots , b_d)$ be a $\R^d$-valued
function on $\R^d$ with $|b|\in \bbK_{d, \alpha-1}$.
For simplicity, sometimes we just denote it as $b\in \bbK_{d, \alpha-1}$.
In this section, we establish the well-posedness of the martingale
problem for $(\sL^b,C_c^{\infty}(\R^d))$.

\medskip

We first recall from Bogdan and Jakubowski \cite{BJ2007}
the construction of {\it a particular} fundamental solution $q^b(t, x, y)$
for non-local operator $\sL^b$ using a perturbation argument.
It is based on the following heuristics: $q^b(t, x, y)$ of $\sL^b$ can be related to
the fundamental solution $p(t, x, y)$ of $\sL$, which is  the transition density of the symmetric stable process $Y$,
 by the following Duhamel's formula:
\begin{equation}\label{e:1.1c}
 q^b (t, x, y)
=p (t, x, y)+
\int^t_0\int_{\R^d}q^b (s, x, z)\, b(z)\cdot\nabla_z p (t-s, z, y)
\rd z \rd s.
\end{equation}
Applying the above formula recursively, one expects   $q^b(t, x,
y)$ to be expressed as an infinite series in terms of $p$ and its
derivatives. Thus we define
$q^b_0(t, x, y)= p(t, x, y)$ and for $k\geq 1$,
\begin{equation}\label{e:1.2}
q^b_k (t, x, y):= \int_0^t \int_{\R^d} q^b_{k-1} (s, x, z)\, b(z) \cdot
\nabla_z p(t-s, z, y) \rd z.
\end{equation}

The following results come from  \cite[Theorem 1, Lemma 15, Lemma
23]{BJ2007} and their proofs.

\begin{prop}\label{P:2.1}
\begin{description}
\item{\rm (i)} There exist constants $T_0>0$ and $c_1>1$
depending
only on $d$, $\alpha$ and
on $b$ only through the rate at which $ M^\alpha_{|b|}(r)$ goes to zero so that $\sum_{k=0}^\infty q^b_k(t, x, y)$
converges locally uniformly on $(0, T_0]\times \R^d \times \R^d$ to a
jointly continuous positive function $q^b(t, x, y)$ and that on $(0,
T_0]\times \R^d \times \R^d$,
    \begin{equation}\label{e:1.5}
c_1^{-1}\left( t^{-d/\alpha} \wedge
\frac{t}{|x-y|^{d+\alpha}}\right)  \le q^b (t, x, y) \le c_1 \left(
t^{-d/\alpha} \wedge \frac{t}{|x-y|^{d+\alpha}}\right).
\end{equation}
Moreover, $\int_{\R^d} q^b(t, x, y) \rd y=1$ for every $t\in (0, T_0]$ and
$x\in \R^d$.

\item{\rm (ii)}  The function $q^b(t, x, y)$ defined in  (i)  can be extended
 uniquely to a jointly continuous positive  function on
 $(0, \infty) \times \R^d \times \R^d$
so that for all $s$, $ t\in (0, \infty)$
and $x$, $y\in \R^d$,
$\int_{\R^d} q^b(t, x, z) dz=1$ and
\begin{equation}\label{e:1.6}
q^b(s+t,x,y)=\int_{\R^d} q^b(s,x,z)q^b(t,z,y) \rd z.
\end{equation}

\item{\rm (iii)} Define
$ T^b_tf(x):= \int_{\R^d} q^b(t,x,y) f(y)\rd y$.
Then for any $f$, $g \in C^\infty_c(\R^d)$, the space of smooth
functions with compact supports,
$$
\lim_{t \downarrow 0} \frac1{t}
\int_{\R^d}  (T^b_tf(x)-f(x)) g(x)\rd x
= \int_{\R^d}(\sL^b f)(x)\, g(x)\rd x.
$$
\end{description}
\end{prop}

\medskip

  Proposition \ref{P:2.1}(iii) indicates
$q^b(t,x,y)$ is a fundamental solution of $\sL^b$ in
distributional sense.
It is  easy to check  (see
\cite[Proposition 2.3]{CKS2012a})
that the   operators $\{T^b_t; t\geq 0\}$ determined by $q^b(t, x, y)$ form a Feller semigroup and
so there exists a $\R^d$-valued  conservative Feller process
$ \{ X_t, t\geq 0, \bbP_x, x\in \R^d\}$ defined on the canonical
 Skorokhod space $\bbD ([0, \infty); \R^d)$ having $q^b(t, x, y)$
 as its transition density function.
Moreover, it is shown in \cite[Theorem 2.5]{CKS2012a}
that $ \bbP_x$ is a solution to    the martingale problem
for  $(\sL^b, C^\infty_c(\R^d))$ with $\bbP_x (X_0=x)=1$.
However in both \cite {BJ2007} and \cite{CKS2012a}, neither the uniqueness
 of  fundamental solution $q^b(t, x, y)$ to $\sL^b$
 nor the uniqueness of the martingale problem for $(\sL^b, C_c^\infty (\R^d))$ are addressed. The main result of this paper, Theorem \ref{T:1.2}, in particular fills
 in this  missing piece and implies that  $q^b(t, x, y)$
 is the transition density function of
 the uniqueness solution $(X, \bbP_x, x\in \R^d)$
to the martingale problem for $(\sL^b, C^\infty_c(\R^d))$.

\medskip

\begin{thm}\label{T:2.2} For each $x\in \R^d$,
 the martingale problem for $(\sL^b, C^\infty_c(\R^d ))$
with initial value $x$ is well-posed.
These martingale
problem solutions $\{ \bbP_x, x\in \R^d\}$ form
a strong Markov process $X^b$,
which has infinite lifetime and possesses a jointly
continuous transition density function $p^b(t, x, y)$ with
respect to the Lebesgue measure on $\R^d$.
Consequently,   $p^b(t, x, y)$ is the same
as the kernel $q^b(t, x, y)$
constructed in Proposition \ref{P:2.1} and
enjoys the two-sided estimates \eqref{e:1.5}.
\end{thm}

Here $X^b_t=X_t$ is the coordinate map defined on $\bbD([0, \infty); \R^d)$ but we use superscript $b$ for emphasis when it is viewed as a Markov
process under probability measures $\bbP_x.$

Let $(\bbP_x, x\in \R^d)$ be the probability measures
on $\bbD ([0, \infty); \R^d)$ obtained from
the kernel $q^b(t, x, y)$ in Proposition \ref{P:2.1}.
The mathematical expectation taken under
 $\bbP_x$ will be denoted by $\bbE_x$.
As we noted in previous paragraph,
for each $x\in \R^d$, $\bbP_x$ solves the martingale problem
for $(\sL^b, C^\infty_c(\R^d))$ with initial value $x$.
We will show that $\bbP_x$ is in fact the
unique solution.
Our approach is motivated by that of
Bass and Chen \cite[Section~5]{BC2003}.

\medskip

Before the proof of Theorem \ref{T:2.2} we state two lemmas on the boundedness of the $\lambda$-resolvent
operator $R_\lambda$ corresponding to
symmetric  $\alpha$-stable process $Y$.
Denote by $p(t, x, y)=p(t, x-y) $ the transition density function of $Y$.
Let
$r_\lambda(x)=\int_0^\infty \e^{-\lambda t} p(t,x) \rd t$ and define
the resolvent operator $R_\lambda$ by
\[ R_\lambda g(x) = \int_{\R^d} r_\lambda(x-y) g(y) \rd y =
\int_{\R^d} r_\lambda(y) g(x-y) \rd y, \]
for every $g\in C_b(\R^d)$ and $x\in \R^d$.
Here $C_b (\R^d)$ (resp. $C_0(\R^d)$) denote the space
of bounded continuous functions on $\R^d$ (resp. continuous
functions on $\R^d$ that vanish at infinity).
For $f\in C_b(\R^d)$, define $\Vert f\Vert_\infty=\sup_{x\in \R^d}
|f(x)|$. Denote by   $C_0^\infty(\R^d)$    the space of
smooth functions on $\R^d$ that together
with their partial derivatives of any order vanish at infinity.
By \cite[Lemma 7]{BJ2007}, for $(\lambda, x)\in (0, \infty)\times \R^d$,
\begin{equation}\label{e:2.1}
 r_\lambda (x) \asymp \left( \lambda^{(d-\alpha)/\alpha}
\vee |x|^{\alpha -d}\right) \wedge \left(\lambda^{-2}|x|^{-d-\alpha}\right),
\end{equation}
 which can be rewritten as
\[ r_\lambda(x) \asymp
\begin{cases}
\frac{1}{|x|^{d-\alpha }} \wedge \frac{\lambda^{-2}}{|x|^{d+\alpha}}
\qquad &\text{when } d>\alpha , \cr
\lambda^{(d-\alpha)/2} \wedge \frac{\lambda^{-2}}{|x|^{d+\alpha}}
&\text{when } d\leq \alpha .
\end{cases}
\]
Here for any two positive functions $f$ and $g$,
$f\asymp g$ means that there is a positive constant $c\geq 1$
so that $c^{-1}\, g \leq f \leq c\, g$ on their common domain of
definition.

\begin{lemma} \label{L:2.3}
  For every $\lambda>0$, $\nabla R_\lambda$ is a
  bounded operator on $C_0(\R^d)$. Moreover, $R_\lambda
  f\in C_0^\infty(\R^d)$ for every $f\in C_0^\infty(\R^d)$.
\end{lemma}

\begin{proof}
It is known by \cite[Lemma 9]{BJ2007} that $r_\lambda (z)$ is continuously
  differentiable off the origin and   there is a constant $c_1 >1$ so that for every $\lambda >0$ and $z\neq0$,
  \begin{equation}\label{e:2.2}
  c_1^{-1} \left(\frac{1}{|z|^{d+1-\alpha}} \wedge
   \frac{1}{\lambda^2 |z|^{d+1+\alpha}}\right)
    \leq | \nabla r_\lambda (z)| \leq c_1 \left( \frac{1}{|z|^{d+1-\alpha}} \wedge \frac{1}{\lambda^2 |z|^{d+1+\alpha}}\right).
   \end{equation}
It follows that for $\lambda>0$, $f\in C_0(\R^d)$ and $x\in \R^d$,
  \[ \int_{\R^d} |\nabla r_\lambda(x-y)|\, |f(y)| \rd y \leq c_1 \Vert
  f\Vert_\infty \int_{\R^d} \left(\frac{1}{|z|^{d-\alpha+1}}\wedge
    \frac{1}{\lambda^2 |z|^{d+\alpha+1}}\right) \rd z <\infty. \]
  Thus $R_\lambda f$ is continuously differentiable and
  \[ \nabla R_\lambda f(x) = \int_{\R^d} \nabla r_\lambda(x-y)
  f(y) \rd y =
  \int_{\R^d} \nabla r_\lambda(y) f(x-y) \rd y. \]
  Since both $r_\lambda (y)$ and $|\nabla r_\lambda(y)|$
  are integrable over $\R^d$ and
  $f(x-y)$ converges to $0$ as $|x|\to\infty$, we conclude that
  both $R_\lambda f$ and $ \nabla R_\lambda f$ are in $C_0(\R^d)$ with $\Vert R_\lambda f\Vert_\infty \leq c_2 \Vert f\Vert_\infty$ and $\Vert\nabla R_\lambda f\Vert_\infty \leq c_2 \Vert f\Vert_\infty$
   for some constant $c_2>0$.
Similarly,
for $f\in C^\infty_0(\R^d)$,
we have
  $$  \partial^{k_1}_{x_1} \cdots \partial^{k_d}_{x_d}  R_\lambda f (x)
      = \int_{\R^d} r_\lambda (y)\, \partial^{k_1}_{x_1} \cdots \partial^{k_d}_{x_d} f(x-y)\, \rd y ,
  $$
and consequently
$R_\lambda f\in C^\infty_0(\R^d)$.
\end{proof}

We may view the function $b$ as a multiplication operator in the sense
that $(b f)(x)= b(x) f(x)$.
\begin{lemma} \label{L:2.4}
  Let $b=(b_1, \cdots, b_d)\in \bbK_{d,\alpha-1}$.
  There exists $\lambda_0>0$
  depending only on $d$, $\alpha$   and on $b$ only via the rate at which $M^\alpha_{|b|}(r)$ goes to zero
such that for every
  $\lambda > \lambda_0$ and $f\in C_0(\R^d)$,
  \[ \Vert \nabla R_\lambda (b f) \Vert_\infty \leq \frac{1}{2} \Vert
  f\Vert_\infty. \]
\end{lemma}

\begin{proof}
It follows from  \cite[Lemma 11 and Corollary 12]{BJ2007} and their proof (with $\beta=2$ there) that
there exists a constant $c_1>0$ depending only on $d$ and $\alpha$
such that for every $t >0$,
\begin{equation}\label{e:2.3}
\sup_{x\in\R^d} \int_{\R^d} \left( \frac{1}{|x-y|^{d+1-\alpha}}
\wedge \frac{t^2}{|x-y|^{d+1+\alpha}}\right) |b(y)| \rd y
 \leq c_1 M^\alpha_{|b|} (t^{1/\alpha}) .
 \end{equation}
 This together with \eqref{e:2.2} implies that there
  exists a constant $c_2>0$ such that for every $\lambda>0$
  \begin{align*}
    \sup_{x\in\R^d} \int_{\R^d} |\nabla
    r_\lambda(x-y)|\,
  |b(y)| \, |f(y)| \rd y & \leq c_2 \Vert f\Vert_\infty \sup_{x\in
    \R^d}\int_{\R^d} \left(\frac{1}{|x-y|^{d-\alpha+1}}\wedge
    \frac{\lambda^{-2}}{|x-y|^{d+\alpha+1}}\right) |b(y)| \rd
  y \\
  & \leq  c_1 c_2 \Vert f\Vert_\infty M_{|b|}^\alpha(\lambda^{-1/\alpha}).
  \end{align*}
    It follows that
  \[ \Vert \nabla R_\lambda(bf) \Vert_\infty \leq c_3
  M_{|b|}^\alpha(\lambda^{-1/\alpha}) \Vert f\Vert_\infty. \]
  Since $M_{|b|}^\alpha(\lambda^{-1/\alpha})$ tends to $0$ as
  $\lambda\to\infty$, there exists some $\lambda_0>0$ so that $c_3 M_{|b|}^\alpha(\lambda^{-1/\alpha}) \leq 1/2$ for
  every $\lambda>\lambda_0$. This proves the lemma.
\end{proof}

It is well known that the transition density function $p(t, x,y)$ of the symmetric $\alpha$-stable process $Y$ on $\R^d$ has the two-sided estimates
$$ p(t, x, y) \asymp t^{-d/\alpha} \wedge \frac{t}{|x-y|^{d+\alpha}}.
$$
 So estimate \eqref{e:1.5} can be restated as
there is a constant $C_1\geq 1$  depending on $b$
only through the rate at which $ M^\alpha_{|b|}(r)$ goes to zero
so that
$$ C_1^{-1} p(t, x, y) \leq p^b(t, x, y) \leq C_1 p(t, x, y)
\qquad \text{for every } (t, x, y)\in (0, T_0]\times \R^d \times \R^d.
$$
It follows from \eqref{e:1.5} and the Chapman-Kolmogorov equation \eqref{e:1.6} that
there are positive constants $C_2\geq 1$ and $C_3>0 $
  depending on $b$
only through the rate at which $ M^\alpha_{|b|}(r)$ goes to zero
so that
\begin{equation}\label{e:2.4}
C_2^{-1} \e^{-C_3t} p(t, x,y)\leq p^b(t, x, y) \leq C_2 \e^{C_3 t}
p(t, x, y) \qquad \text{for every } t>0 \text{ and } x,\ y\in \R^d.
\end{equation}
Thus for $\lambda >C_3$,
$$ \bbE_x\left[ \int_0^\infty \e^{-\lambda t} |b(X_t)| \rd t
  \right] \leq C_2  \int_{\R^d} |b(y)| r_{\lambda-C_3} (x-y) \rd y  .
$$
By \cite[Lemma 16]{BJ2007},
there is a constant $C_4>C_3$ depending on $b$
only through the rate at which $ M^\alpha_{|b|}(r)$ goes to zero
such that for every $\lambda \geq C_4$,
\begin{equation}\label{e:2.5}
\sup_{x\in \R^d} R_\lambda |b| (x) =
 \sup_{x\in \R^d}  \bbE_x\left[ \int_0^\infty
 \e^{-\lambda t} |b(X_t)| \rd t \right] <\infty .
\end{equation}
By increasing the value of $\lambda_0$ in Lemma \ref{L:2.4}
if needed, we may and do assume
that $\lambda_0\geq C_4$.

\begin{thm}  \label{T:2.5}
  For each $x\in\R^d$, $\bbP_x$ is the unique solution to
  the martingale problem for $(\sL^{b},C_c^\infty(\R^d))$ with
  initial value $x$.
\end{thm}

\begin{proof} Let $\bfQ$ be any solution to the martingale problem for
$(\sL^{b}, C^\infty_c(\R^d))$ with initial value $x$.
We will show $\bfQ=\bbP_x$. We divide its proof into 5 steps.

(i) We show that it suffice to consider the case that
\begin{equation}
  \label{e:2.6}
  \bbE_\bfQ\left[ \int_0^\infty \e^{-\lambda t}|b(X_t)| \rd t \right]
  <\infty \quad \text{for every } \lambda>\lambda_0,
\end{equation}
where $\bbE_\bfQ$ is the mathematical
expectation under the probability measure $\bfQ$ and
 $\lambda_0$ is the constant in Lemma \ref{L:2.4}.

 By the definition of martingale problem solution,
$  \int_0^t b(X_s)\cdot\nabla f(X_s) \rd s$
is well defined $\bfQ$-a.s. for each $t>0$, that is,
$\int_0^t |b(X_s)\cdot \nabla f(X_s)| \rd s<\infty$
$\bfQ$-a.s. for every $t>0$.
Let
$$
T_n (f) = \inf \left\{t>0: \int_0^t |b(X_s)\cdot \nabla f(X_s)| \rd s \geq n \right\}.
$$
Then $\{T_n (f), n\geq 1\}$ is an increasing sequence of stopping times
 such that
$\lim_{n\to\infty} T_n (f) =\infty$ $\bf Q$-a.s. with
\begin{equation}
  \label{e:2.7}
  \bbE_\bfQ\left[ \int_0^{T_n} |b(X_s)\cdot \nabla f(X_s)| \rd s\right] \leq  n.
\end{equation}

Choose a sequence of functions $f_n^{(i)}\in C_c^\infty(\R^d)$ such
that $f_n^{(i)}(x)=x_i$ for $x\in B(0,n)$ and $1\leq i \leq d$.
Define
\[
 S_n = \left( \min_{1\leq i\leq d} T_n (f_n^i)\right)
  \wedge \inf\left\{t:\ |X_t| > n \text{ or }
   |X_{t-}| > n \right\}.
   \]
Then $S_n$ is an increasing sequence of stopping times with
$\lim_{n\to\infty} S_n = \infty$. By \eqref{e:2.7},
\begin{equation}
  \label{e:2.8}
  \bbE_\bfQ\left[ \int_0^{S_n} |b(X_s)| \rd s \right]
  \leq \sum_{i=1}^d \bbE_\bfQ\left[ \int_0^{S_n} |b_i(X_s)|
  \rd s \right] \leq \sum_{i=1}^d \bbE_\bfQ\left[\int_0^{S_n} | b(X_s)\cdot\nabla f_n^{(i)}(X_s)| \rd s\right]
  \leq n d.
\end{equation}

Now we construct a new probability measure $\widetilde \bfQ$ so that
$\widetilde
\bfQ$ is also a solution to the martingale problem for $(\sL^{b}, C_c^\infty (\R^d))$
and that for every $\lambda>\lambda_0$,
\[
\bbE_{\widetilde \bfQ} \left[ \int_0^\infty \e^{-\lambda t} |b(X_t)| \rd t
\right] <\infty.
 \]
Let $\calF_t$ be the minimal filtration generated by $\{X_s; s\leq t\}$.
Fix $N\geq 1$. We specify $\widetilde \bfQ$ by
\[ \widetilde \bfQ \left(B\bigcap (C\circ \theta_{S_N})\right) =
\bbE_\bfQ \left[\bbP_{X_{S_N}}(C);\ B\right], \]
for $B\in \calF_{S_N}$ and $C\in \calF_\infty$. It
is easy to see that $\widetilde \bfQ$ is again a solution to
the martingale problem for $(\sL^{b}, C^\infty_c (\R^d))$. Moreover,
\[ \bbE_{\widetilde \bfQ}\left[ \int_0^\infty \e^{-\lambda t}|b(X_t)|
  \rd t \right]= \bbE_\bfQ\left[ \int_0^{S_N} \e^{-\lambda t} |b(X_t)|
  \rd t\right] + \bbE_{\bfQ}\left[\e^{-\lambda
    S_N}\bbE_{X_{S_N}}\left[\int_0^\infty \e^{-\lambda t}|b(X_t)|\rd
    t\right]\right],  \]
which is finite by \eqref{e:2.5} and \eqref{e:2.8}. Since
$\wt \bfQ=\bfQ$ on $\calF_{S_N}$, if we can show $\wt \bfQ=\bbP_x$,
then this would imply $\bfQ= \bbP_x$ on $\calF_{S_N}$. Since $N\geq 1$
 is arbitrary, this would imply that $\bfQ=\bbP_x$
on $\calF_\infty$. So it suffice
to consider the solution $\bfQ$ to the martingale problem satisfying
\eqref{e:2.6}.

\medskip

(ii) We next show that for every $g\in C^\infty_0 (\R^d)$,
\begin{equation}
  \label{e:2.9}
  \bbE_\bfQ\left[\int_0^\infty \e^{-\lambda t} g(X_t)\rd t\right] =
  R_\lambda g(x) + \bbE_\bfQ\left[\int_0^\infty \e^{-\lambda t}
    b(X_t)\cdot\nabla R_\lambda g(X_t) \rd t\right].
\end{equation}

By \eqref{e:1.7}, $f(X_t)$ is a semimartingale under $\bfQ$ for
every $f\in C_c^\infty(\R^d)$. It follows by the It\^o's formula that
\begin{align*} &\e^{-\lambda t} f(X_t)= f(X_0) + \int_0^t
  \e^{-\lambda
  s} \rd M_s^f + \int_0^t \e^{-\lambda s} \left(\Delta^{\alpha/2} f(X_s) +
b(X_s)\cdot
\nabla f(X_s)\right)\rd s - \lambda \int_0^t \e^{-\lambda s}
f(X_s) \rd s.
\end{align*}
Taking expectation with respect to $\bfQ$, we have
\begin{equation}
  \label{e:2.10}
  \bbE_\bfQ[\e^{-\lambda t}f(X_t)]= f(x)- \bbE_\bfQ\left[ \int_0^t
  \e^{-\lambda s}(\lambda f-\Delta^{\alpha/2} f)(X_s)\rd s\right] +
\bbE_\bfQ\left[
  \int_0^t \e^{-\lambda s}b(X_s)\cdot \nabla f(X_s)\rd s\right].
\end{equation}
Note that $f$, $\nabla f$ and $\Delta^{\alpha/2} f$ are all bounded. Taking limit
$t\to\infty$ in both sides of \eqref{e:2.10} and using the
fact \eqref{e:2.6}, we obtain
\begin{equation}
  \label{e:2.11}
  \bbE_\bfQ\left[\int_0^\infty \e^{-\lambda t}(\lambda f-\Delta^{\alpha/2}
    f)(X_t)\rd t \right] = f(x) + \bbE_\bfQ\left[\int_0^\infty
    \e^{-\lambda t}
    b(X_t)\cdot\nabla f(X_t)\rd t \right].
\end{equation}
We want to show that \eqref{e:2.11} holds for all $f\in
C_0^\infty(\R^d)$. In fact, for any $f\in
C_0^\infty(\R^d)$, there exists a sequence of functions $f_n \in
C_c^\infty(\R^d)$ such that $f_n\to f$ in $C_0^\infty(\R^d)$ and
in particular, $\Vert f_n-f\Vert_\infty\to0$, $\Vert \Delta^{\alpha/2}
f_n -\Delta^{\alpha/2} f\Vert_\infty\to0$, $\Vert
\nabla f_n - \nabla f\Vert_\infty\to0$. Applying \eqref{e:2.6} again,
we have
\[ \lim_{n\to\infty} \bbE_\bfQ\left[\int_0^\infty \e^{-\lambda t}(\lambda f_n - \Delta^{\alpha/2}
  f_n)(X_t)\rd t\right] = \bbE_\bfQ\left[\int_0^\infty
  \e^{-\lambda t}(\lambda f-\Delta^{\alpha/2} f)(X_t)\rd t \right], \]
and
\[ \lim_{n\to\infty} \bbE_\bfQ\left[\int_0^\infty
  \e^{-\lambda t}
  b(X_t)\cdot\nabla f_n(X_t)\rd t \right] =
\bbE_\bfQ\left[\int_0^\infty \e^{-\lambda t}
  b(X_t)\cdot\nabla f(X_t)\rd t \right]. \]
Thus \eqref{e:2.11} holds for $f\in C_0^\infty(\R^d)$.

By Lemma \ref{L:2.3}, $R_\lambda g\in C_0^\infty(\R^d)$ for
$g\in C_0^\infty(\R^d)$. Taking $f=R_\lambda g$ in
\eqref{e:2.11} and using the fact $(\lambda -\Delta^{\alpha/2}) R_\lambda
g=g$, we obtain \eqref{e:2.9}.

\medskip

(iii) We claim that
\begin{equation}\label{e:2.12}
  \bbE_\bfQ\left[\int_0^\infty \e^{-\lambda t} \id_A(X_t)\rd t\right] =0
 \qquad \text{for any } A\subset \R^d \text{ with } m(A)=0,
\end{equation}
where $m$ is the Lebesgue measure on $\R^d$.

To see this, suppose $A$ is a bounded subset of $\R^d$
having $m(A)=0$. Let $\psi_n$ be a sequence of positive
functions in $C^\infty_c(\R^d)$ so that $|\psi_n|\leq 2$,
$\lim_{n\to \infty}\psi_n=0$
$m$-a.e. on $\R^d$ and $\lim_{n\to \infty}\psi_n\geq \id_A$.
It follows from  \eqref{e:2.2} and the dominated convergence
theorem that
$\nabla R_\lambda \psi_n (z) = \int_{\R^d} \nabla r_\lambda (z-y)
\psi_n (y) \rd y$ converges to 0 boundedly as $n\to \infty$.
One concludes then from \eqref{e:2.6} and the dominated convergence
theorem that
$$ \lim_{n\to \infty} \bbE_\bfQ\left[\int_0^\infty \e^{-\lambda t}
    b(X_t)\cdot\nabla R_\lambda \psi_n(X_t) \rd t\right]
 =0.
$$
Applying Fatou's lemma to \eqref{e:2.9} with $\psi_n$ in place of $g$ yields that
\[ \bbE_\bfQ\left[\int_0^\infty \e^{-\lambda t} \id_A(X_t)\rd t\right]
 \leq  \liminf_{n\to \infty}
 \bbE_\bfQ\left[\int_0^\infty \e^{-\lambda t} \psi_n(X_t)\rd t\right]
 =  \liminf_{n\to \infty}R_\lambda \psi_n (x) =0,
 \]
where the last equality is due to \eqref{e:2.1} and the dominated convergence theorem. This establishes \eqref{e:2.12} for any bounded
and hence for any subset $A\subset \R^d$ having $m(A)=0$.

\medskip

(iv) We now show that \eqref{e:2.9} holds for any function $g$ on
$\R^d$ with $|g|\leq c |b|$ as well.

Let $g$ be a function on $\R^d$ with $|g|\leq c |b|$ for some $c>0$.
Fix $M>0$ and define $g_M=\left( ((-M)\vee g)\wedge M\right) \id_{B(0, M)}$.
Let $\phi$ be a positive smooth function on $\R^d$ with compact
support such that $\int_{\R^d} \phi(y) \rd y=1$. For $n\geq 1$,
set $\phi_n(y)= n^d \phi(n y)$ and $f_{n} (z):= \int_{\R^d} \phi_n(z-y) g_M(y) \rd y$. Then $f_{n}\in C^\infty_c(\R^d)$,
$|f_{n}|\leq M$,
and $f_{n}$ converges to $g_M$ almost everywhere on $\R^d$
as $n\to \infty$.
In view of \eqref{e:2.12} and the bounded convergence theorem,
$\lim_{n\to \infty} R_\lambda f_{ n}(x)= R_\lambda g_{M}(x)$
and
$$
\lim_{n\to \infty} \bbE_\bfQ\left[\int_0^\infty
  \e^{-\lambda t} f_{n}(X_t)\rd t\right]
  = \bbE_\bfQ\left[\int_0^\infty \e^{-\lambda t} g_{M}(X_t)\rd t\right].
$$
On the other hand, by \eqref{e:2.2} and the dominated convergence theorem,
$\nabla R_\lambda f_n$ converges boundedly on $\R^d$
to $\nabla R_\lambda g_M$
as $n\to \infty$. So we deduce from \eqref{e:2.9}
with $f_n$ in place of $g$ and take $n\to \infty$ that
\begin{equation} \label{e:2.13}
  \bbE_\bfQ\left[\int_0^\infty \e^{-\lambda t} g_M(X_t)\rd t\right] =
  R_\lambda g_M(x) + \bbE_\bfQ\left[\int_0^\infty \e^{-\lambda t}
    b(X_t)\cdot\nabla R_\lambda g_M(X_t) \rd t\right].
\end{equation}

Clearly by the dominated convergence theorem, \eqref{e:2.1} and
\eqref{e:2.6},
$$ \lim_{M\to \infty} \bbE_\bfQ\left[\int_0^\infty \e^{-\lambda t} g_M(X_t)\rd t\right] = \bbE_\bfQ\left[\int_0^\infty \e^{-\lambda t} g(X_t)\rd t\right] \quad \text{and} \quad
\lim_{M\to \infty} R_\lambda g_M(x) = R_\lambda g (x),
$$
while in view of \eqref{e:2.2} and \eqref{e:2.3},
$\nabla R_\lambda g_M (z) = \int_{\R^d} \nabla r_\lambda(z-y) g_M (y)
 \rd y$ converges boundedly on $\R^d$ to $\nabla R_\lambda g (z)$.
Thus by \eqref{e:2.6} and the dominated convergence theorem,
$$ \lim_{M\to \infty} \bbE_\bfQ\left[\int_0^\infty \e^{-\lambda t}
    b(X_t)\cdot\nabla R_\lambda g_M(X_t) \rd t\right]
    = \bbE_\bfQ\left[\int_0^\infty \e^{-\lambda t}
    b(X_t)\cdot\nabla R_\lambda g (X_t) \rd t\right].
$$
The last two displays together with \eqref{e:2.13}
 establish the claim that \eqref{e:2.9} holds
for any $g$ with $|g| \leq c |b|$.

\medskip

(v) Define a linear functional $V_\lambda$ by
\[ V_\lambda f = \bbE_\bfQ\left[ \int_0^\infty \e^{-\lambda t} f(X_t) \rd
t\right]. \]
Then \eqref{e:2.9} can be rewritten as
\begin{equation}
  \label{e:2.14}
  V_\lambda g = R_\lambda g(x) + V_\lambda (BR_\lambda g)
  \quad \text{for } g\in C^\infty_0(\R^d)\bigcup\{ g: |g| \leq c |b|
  \text{ for some }c>0\},
\end{equation}
where $B$ is the operator defined by
\[ Bf(x) = b(x) \cdot \nabla f(x). \]
Fix $g\in C_c^\infty (\R^d)$. It follows from \eqref{e:2.2} that
$| BR_\lambda g| \leq c |b|$ for some constant $c>0$. Applying \eqref{e:2.14}
with  $B R_\lambda g$ in place of $g$ yields
\begin{equation}
  \label{e:2.15}
  V_\lambda(B R_\lambda g) = R_\lambda (B R_\lambda g)(x) + V_\lambda(B
  R_\lambda B R_\lambda g).
\end{equation}
Repeating this procedure, we get that for every
$g\in C^\infty_c(\R^d)$ and every integer $N\geq 1$,
\begin{equation}
  \label{e:2.16}
  V_\lambda g = \sum_{k=0}^N R_\lambda (B R_\lambda)^k g(x) +
  V_\lambda(B( R_\lambda B)^{N}R_\lambda g).
\end{equation}

It follows from  Lemma \ref{L:2.4} that for $\lambda > \lambda_0$,
\begin{align*}
  |V_\lambda(B(R_\lambda B)^N R_\lambda g)|
  \leq & \Vert (\nabla
R_\lambda b)^N \nabla R_\lambda g\Vert_\infty \bbE_\bfQ\left[
  \int_0^\infty \e^{-\lambda t} |b(X_t)| \rd t\right] \\
\leq & 2^{-N} \Vert
\nabla R_\lambda g\Vert_\infty \bbE_\bfQ\left[
  \int_0^\infty \e^{-\lambda t} |b(X_t)| \rd t\right],
\end{align*}
which tends to $0$ as $N\to\infty$. Passing  $N\to\infty$ in
\eqref{e:2.16} gives
\begin{equation}
  \label{e:2.17}
  V_\lambda g = \sum_{k=0}^\infty R_\lambda (B R_\lambda)^k g(x).
\end{equation}

Note that $\bbP_x$ is also a solution to the martingale problem for
$(\sL^{b},C_c^\infty(\R^d))$ with initial value $x$. Then \eqref{e:2.17} also holds
with $\bfQ$ replaced by $\bbP_x$, that is,
\[
 \bbE_x\left[\int_0^\infty \e^{-\lambda t} g(X_t)\rd t\right] =
\sum_{k=0}^\infty R_\lambda (B R_\lambda)^k g(x).
 \]
Consequently
\begin{equation}
  \label{e:2.18}
  \bbE_\bfQ\left[\int_0^\infty \e^{-\lambda t} g(X_t)\rd t\right] =
  \bbE_x\left[\int_0^\infty \e^{-\lambda t} g(X_t)\rd t\right]
\end{equation}
for every $g\in C^\infty_c(\R^d)$ and $\lambda>\lambda_0$.
By the uniqueness
of the Laplace transform, we have $\bbE_\bfQ[g(X_t)]=\bbE_x[g(X_t)]$ for
all $t$, or, the one-dimensional distributions of $X_t$ under $\bfQ$
and $\bbP_x$ are the same. By a standard argument using regular conditional probability (see, e.g., the proof of Theorem VI.3.2 in \cite{B}), one
 obtains equality of all finite-dimensional distributions and hence $\bfQ =
\bbP_x$. The uniqueness for the martingale problem for $(\sL^{b}, C^\infty_c(\R^d))$ is thus proved.
\end{proof}

\medskip

\noindent{\bf Proof of Theorem \ref{T:2.2}.}
The existence and uniqueness for the martingale problem for $(\sL^b, C^\infty_c(\R^d)$ is established in Theorem  \ref{T:2.5}.
By the uniqueness, the remaining assertions then follow from Proposition \ref{P:2.1}. \qed

\section{Stochastic differential equation}\label{S:3}

It is known that for any $\alpha\in (0, 2)$ the fractional Laplacian
$\Delta^{\alpha/2}$
can be written in the form
\begin{equation}\label{e:3.2}
 \Delta^{\alpha /2} u(x)\, =\,  \int_{\R^d}
  \big( u(x+z)-u(x)-\nabla u(x)\cdot z
\id_{\{|z|\leq 1\}} \big)  \frac{{\cal A}(d, - \alpha)
}{|z|^{d+\alpha}}\, \rd z,
\end{equation}
where ${\cal A}(d, - \alpha)$ is a normalizing constant so that
\begin{equation}\label{e:3.2b}
\int_{\R^d} \left( \e^{\im\xi \cdot z} -1-\im\xi \cdot z
\id_{\{|z|\leq 1\}}\right)
\frac{{\cal A}(d, -\alpha)}{|z|^{d+\alpha}} \rd z
= -|\xi |^\alpha, \qquad \xi \in \R^d.
\end{equation}
In fact, ${\cal A}(d, -\alpha)$ can be computed explicitly in
terms of $\Gamma$-function:
$$
{\cal A}(d, - \alpha)= \alpha2^{\alpha-1}\pi^{-d/2} \Gamma(\frac{d+\alpha}2) \Gamma(1-\frac{\alpha}2)^{-1}.
$$
When $\alpha \in (1, 2)$ as is assumed in this paper,
the $\id_{\{|z|\leq 1\}}$ term can be dropped from both
\eqref{e:3.2} and \eqref{e:3.2b}.
It is also known that the symmetric $\alpha$-stable process $Y$
has L\'evy intensity function
\begin{equation}\label{e:3.1}
J(x, y)={\cal A}(d, -\alpha)|x-y|^{-(d+\alpha)}.
\end{equation}
The L\'evy intensity function gives rise to a L\'evy system $(N, H)$
for $X$, where $N(x, \rd y)=J(x, y)\rd y$ and $H_t=t$, which describes the
jumps of the process $Y$.

Recall that $(X^b, \bbP_x, x\in \R^d)$
is the unique solution to the martingale problem for $(\sL^b,
C^\infty_c(\R^d))$ on the canonical Skorokhod space $\Omega:=\bbD([0,
\infty); \R^d)$. The following theorem shows that $X^b$ is a weak
solution to the SDE \eqref{e:1.8}.

\begin{thm}\label{T:3.1}
There exists a process $Z$ defined on $\Omega$ so that all its paths are right
continuous and admit left limits (rcll), and
that under each $\bbP_x$, $Z$ is a rotationally symmetric $\alpha$-stable process on $\R^d$ and
\begin{equation}\label{e:3.3}
X^b_t=x +Z_t + \int_0^t b(X^b_s) \rd s, \quad t\geq 0.
\end{equation}
\end{thm}

\begin{proof} By Theorem \ref{T:2.2}, $X^b$ is the same as the Feller process determined by kernel $q^b(t, x, y)$ in Proposition \ref{P:2.1}.
Observe that it follows from \eqref{e:2.5} that $\bbE_x\left[ \int_0^t
|b(X^b_s)| \rd s\right] <\infty$ for every $t>0$.
Let $\{\calF_t; t\geq 0\}$
be the minimal augmented filtration generated by $X^b_t$.
We know from \cite[Theorem 2.6]{CKS2012a} that $X^b$ has the same
L\'evy system $(J(x,y)\rd y, t)$ as that of symmetric $\alpha$-stable process;
that is,  for any
$x\in \R^d$,  non-negative measurable function $f$ on
$\R^d\times \R^d$ vanishing along the diagonal $\{( x, y)\in
\R^d\times \R^d: x=y\}$, predictable process $\xi_t$ and stopping time $T$ with respect to  the filtration $\{\calF_t; t\geq 0\}$,
\begin{equation}\label{e:3.4}
\bbE_x \Big[ \sum_{s\le T} \xi_s f( X^b_{s-}, X^b_s) \Big]
= \bbE_x \left[ \int_0^T \xi_s  \left( \int_{\R^d} f(X^b_s, y)
J(X^b_s, y)\rd y \right) \rd s \right].
 \end{equation}
In particular, if $\sum_{s\leq t} f(X_{s-}^b,X_s^b)$ has
$\bbP_x$-integrable variation, then
\[ \int_0^t \left(\int_{\R^d} f(X_s^b,y)J(X_s^b,y)\rd y\right) \rd
s \]
is its dual predictable projection, that is,
\begin{equation}
  \label{e:3.4b}
  \sum_{s\leq t} f(X_{s-}^b,X_s^b) - \int_0^t \left(\int_{\R^d}
    f(X_s^b,y)J(X_s^b,y)\rd y\right) \rd s
\end{equation}
is a $\bbP_x$-martingale
(cf. \cite[Definition 5.21 and Corollary 5.31]{HWY}).

It follows from (\ref{e:3.4}) that
$\sum_{s\in [0, t]} |X^b_s -X^b_{s-}|\id_{\{ |X^b_s-X^b_{s-}|\geq 1\}}$ is
$\bbP_x$-integrable and so
$$ M^{d, 1}_t:= \sum_{s\in [0, t]} (X^b_s -X^b_{s-}) \id_{\{ |X^b_s-X^b_{s-}|\geq 1\}}
$$
 is a $\bbP_x$-martingale. Moreover,
$$ M^{d, 2}_t:= \lim_{\eps \to 0} \sum_{s\in [0, t]} (X^b_s -X^b_{s-}) \id_{\{ \eps<|X^b_s-X^b_{s-}|<1\}}
$$
is a purely discontinuous $\bbP_x$-square-integrable  martingale
with $M^{d, 2}_t-M^{d, 2}_{t-}=(X^b_t -X^b_{t-}) \id_{\{|X^b_t-X^b_{t-}|<1\}}$.
Define $Z_t=M^{d, 1}_t+M^{d, 2}_t$, which is a martingale under each
$\bbP_x$ with
\begin{equation}\label{e:3.5}
Z_t-Z_{t-}=X^b_{t}-X^b_{t-}\quad \text{for } t>0.
\end{equation}
Since $(X^b, \bbP_x)$ solves the martingale problem,
$X^b$ is a semimartingale. In view of \eqref{e:3.5}, it
 can be uniquely expressed as
\begin{equation}\label{e:3.6}
 X^b_t=X^b_0+ M_t+ Z_t +A_t,
\end{equation}
where $M=(M^1, \cdots, M^d)$ is a continuous
local martingale and $A$ is a  continuous
process of finite variation. We will use $\< M^i, M^j\>$ to denote
the quadratic covariation of $M^i$ and $M^j$.
 For $f\in C^\infty_c(\R^d)$, applying Ito's formula to \eqref{e:3.6}
 (cf. \cite[Theorem 9.35]{HWY})
 and using the L\'evy system for $X^b$ and \eqref{e:3.2},
 we have
\begin{align*}
 &  f(X^b_t)-f(X^b_0)\\
  =&  \int_0^t \nabla f (X_{s-}) \rd (M_s+ Z_s) + \int_0^t \nabla f(X_{s}) \rd A_s  + \frac12 \sum_{i, j=1}^d \int_0^t \frac{\partial^2 f}
  {\partial x_i \partial x_j} (X^b_s) \rd \< M^i, M^j\>_s \\
&   \qquad + \sum_{s\leq t} \left( f(X^b_s)-f(X^b_{s-})-\nabla f(X^b_{s-}) \cdot
   (X^b_s -X^b_{s-})\right) \\
  =&\text{ local martingale } + \int_0^t \nabla f(X^b_s) \rd A_s
  + \frac12 \sum_{i, j=1}^d \int_0^t \frac{\partial^2 f}
  {\partial x_i \partial x_j} (X^b_s) \rd \< M^i, M^j\>_s \\
& \qquad +
  \int_0^t  \int_{\R^d}\left(f(X^b_s+z)-f(X^b_s)-\nabla f(X^b_s) \cdot z\right) \frac{{\cal A}(d, -\alpha)} {|z|^{d+\alpha}}
   \, \rd z \,  \rd s \\
 =& \text{ local martingale } + \int_0^t \nabla f(X^b_s) \rd A_s
  + \frac12 \sum_{i, j=1}^d \int_0^t \frac{\partial^2 f}
  {\partial x_i \partial x_j} (X^b_s) \rd \< M^i, M^j\>_s
  + \int_0^t \Delta^{\alpha/2} f(X^b_s) \rd s.
\end{align*}
Since $(X^b_t, \bbP_x)$ solves the martingale problem for
$(\sL^b, C^\infty_c(\R^d))$,
\[ M^f_t:=f(X^b_t)-f(X^b_0)-\int_0^t
\sL^b f(X^b_s) \rd s \]
is a martingale and so we conclude
$$
 \int_0^t b(X^b_s) \cdot \nabla f(X^b_s) \rd s
= \int_0^t \nabla f(X^b_s) \rd A_s
  + \frac12 \sum_{i, j=1}^d \int_0^t \frac{\partial^2 f}
  {\partial x_i \partial x_j} (X^b_s) \rd \< M^i, M^j\>_s .
$$
Since the above holds for every $f\in C^\infty_c (\R^d)$, we must have
$$ A_t =\int_0^t b(X^b_s) \rd s \quad \text{and} \quad
 \<M^i, M^j\>_t =0 \text{ for every } 1\leq i, j\leq d.
$$
Hence $M=0$ and so by \eqref{e:3.6}, $X^b_t=X^b_0+Z_t + \int_0^t
b(X^b_s) \rd s$.

It remains to show that $Z$ is a rotationally symmetric $\alpha$-stable process under $\bbP_x$.
For $\xi \in \R^d$, applying Ito's formula
for $f(x)=\e^{\im\xi \cdot x}$ to   martingale $Z_t$ and using
L\'evy system formula \eqref{e:3.4}, we get
\begin{align*}
\bbE_x [ \e^{\im\xi \cdot Z_t} ]
=& 1 + \bbE_x \left[ \sum_{s\leq t} \left( \e^{\im\xi\cdot Z_s}
-\e^{\im\xi\cdot Z_{s-}}-  \e^{\im\xi \cdot Z_{s-}}\, \im \xi
\cdot (Z_s -Z_{s-})\right)\right] \\
=& 1 + \bbE_x \left[ \sum_{s\leq t}   \e^{\im\xi\cdot Z_{s-}}
 \left( \e^{\im\xi \cdot (X^b_s-X^b_{s-})}-1 - \im \xi
\cdot (X^b_s -X^b_{s-})\right)  \right] \\
=& 1 + \bbE_x \left[ \int_0^t \int_{\R^d}
  \e^{\im\xi\cdot Z_{s-}} \left(\e^{\im\xi \cdot z}-1-  \im \xi
\cdot z \right) \frac{{\cal A}(d, -\alpha)}{|z|^{d+\alpha}}
\rd z \rd s \right] \\
=& 1 + \bbE_x \left[ \int_0^t \int_{\R^d}
  \e^{\im\xi\cdot Z_{s}} \left(\e^{\im\xi \cdot z}-1-  \im \xi
\cdot z \id_{\{|z|\leq 1\}} \right) \frac{{\cal A}(d, -\alpha)}{|z|^{d+\alpha}}
\rd z  \rd s \right] \\
=& 1- | \xi|^\alpha
\int_0^t    \bbE_x \left[ \e^{\im\xi\cdot Z_{s}} \right] \rd s ,
\end{align*}
where in the last equality, we used \eqref{e:3.2b} and Fubini's theorem.
Set $\phi (t)= \bbE_x [ \e^{\im\xi \cdot Z_t} ]$.
We see from above that
  $\phi (t)=1 -|\xi|^\alpha
\int_0^t \phi (s) \rd s$. Differentiate in $t$, one solves easily
that $\phi (t)=\e^{-t |\xi|^\alpha}$.
Now by the Markov property of $X^b$, we have for every $s$, $t>0$,
\begin{align*}
\bbE_x \left[ \e^{\im\xi \cdot (Z_{t+s}-Z_t)} \big| \calF_t \right]
=& \bbE_{X^b_t} \left[ \e^{\im\xi \cdot Z_s} \right] = \e^{-s |\xi|^\alpha}.
\end{align*}
This proves that, under each $\bbP_x$,
$Z$ is a process having independent stationary increments
and its characteristic function is $\e^{-s |\xi|^\alpha}$;
that is, $Z$ is a rotationally symmetric $\alpha$-stable process on $\R^d$.
\end{proof}

\medskip

Now we are ready to complete the proof for the uniqueness of weak
solution to the SDE \eqref{e:1.8}.
\medskip

\noindent{\bf Proof of Theorem \ref{T:1.2}.}
Fix $x\in \R^d$.
Theorem \ref{T:3.1} gives the existence of a weak solution to SDE
\eqref{e:1.8}.
Using Ito's formula, every weak solution to \eqref{e:1.8} solves
the martingale problem for $(\sL^b, C^\infty_c (\R^d))$.
So the uniqueness of weak solution to \eqref{e:1.8} follows
from the uniqueness of the martingale problem for
$(\sL^b, C^\infty_c (\R^d))$, which is established in Theorem
\ref{T:2.2}. \qed

\medskip

\noindent{\bf Proof of Corollary \ref{C:1.3}.}
By Theorem \ref{T:2.2}, $\alpha$-stable process $X^b$ with drift $b$
is the Feller process with transition density function $q^b(t,x, y)$
constructed in Proposition \ref{P:2.1}. The conclusion of the Corollary
follows readily from Proposition \ref{P:2.1} and
\cite[Theorem 1.3]{CKS2012a}. \qed

\bibliographystyle{abbrv}

\begin{thebibliography}{99}


\bibitem{B}
R.~F. Bass.  {\it Diffusions and Elliptic Operators}.
Springer, 1998.

\bibitem{BC2003}
R.~F. Bass and Z.-Q. Chen.
\newblock Brownian motion with singular drift.
\newblock {\em Ann. Probab.} {\bf 31} (2003), 791--817.



\bibitem{BJ2007}
K.~Bogdan and T.~Jakubowski.
\newblock Estimates of heat kernel of fractional {L}aplacian perturbed by
  gradient operators.
\newblock {\em Comm. Math. Phys. \bf 271} (2007), 179--198.


\bibitem{CKS2012a}
Z.-Q. Chen, P.~Kim, and R.~Song.
\newblock {D}irichlet heat kernel estimates for fractional {L}aplacian with
  gradient perturbation.
{\it Ann. Probab. \bf 40} (2012), 2483--2538.

\bibitem{CS} Z.-Q. Chen and R. Song,
Drift transforms and Green function estimates for discontinuous processes.
{\it J. Funct. Anal. \bf 201} (2003), 262-281.


\bibitem{HWY}  S. W. He, J. G. Wang and J. A. Yan,
{\sl Semimartingale Theory and
Stochastic Calculus}. Science Press, Beijing New york, 1992.

\bibitem{KS} P. Kim and R. Song, Stable process with singular drift.
Preprint 2013.


\bibitem{K} T.~G.~Kurtz. Equivalence of stochastic equation and martingale problem,. D. Crisan (ed,) {\it Stochastic Analysis 2010}, 113--130.
    Springer-Verlag 2011.

\end{thebibliography}

\vskip 0.3truein

{\bf Zhen-Qing Chen}\\
\indent Department of Mathematics, University of Washington, Seattle, WA
98195, USA \\
\indent Email: \texttt{zqchen@uw.edu}

\vspace{1em}
{\bf Longmin Wang}\\
\indent School of Mathematical Sciences, Nankai University, Tianjin 300071,
P. R. China \\
\indent Email: \texttt{wanglm@nankai.edu.cn}

\end{document}